\documentclass{amsart}

\newtheorem{theorem}{Theorem}[section]
\newtheorem{lemma}[theorem]{Lemma}
\usepackage{amssymb}
\usepackage{amsmath}
\usepackage{stmaryrd}
\usepackage{hyperref}
\hypersetup{colorlinks=true,
linkcolor=red,
filecolor=magneta,
}
\usepackage{blindtext}

\theoremstyle{definition}
\newtheorem{definition}[theorem]{Definition}
\newtheorem{example}[theorem]{Example}

\newtheorem{proposition}[theorem]{Proposition}
\newtheorem{corollary}[theorem]{Corollary}
\newtheorem{remark}[theorem]{Remark}
\DeclareMathOperator{\Cone}{Cone}
\DeclareMathOperator{\Max}{max}
\DeclareMathOperator{\Vol}{Vol}

\DeclareMathOperator{\Dim}{dim}

\numberwithin{equation}{section}

\begin{document}

\title{A Volume = Multiplicity formula for $p$-families of ideals }

\author{Sudipta Das}
\address{1290 Frenger Mall, Las Cruces, New Mexico 88003-8001}
\curraddr{}
\email{sudiptad@nmsu.edu}
\thanks{}

\date{}

\dedicatory{}


\begin{abstract}
In this paper, we work with certain families of ideals called $p$-families in rings of prime characteristic. This family of ideals is present in the theories of tight closure, Hilbert-Kunz multiplicity, and $F$-signature. For each $p$-family of ideals, we attach a Euclidean object called $p$-body, which is analogous to the Newton Okounkov body associated with a graded family of ideals. Using the combinatorial properties of $p$-bodies and algebraic properties of the Hilbert-Kunz multiplicity, we establish in this paper a Volume = Multiplicity formula for $p$-families of $\mathfrak{m}_{R}$-primary ideals in a Noetherian local ring $R$.
\end{abstract}

\maketitle

\section{Introduction}
Let $(R,\mathfrak{m}, \mathbb{K})$ be a Noetherian local ring of dimension $d$, with prime characteristic $p>0$, $I$ be an $\mathfrak{m}$-primary ideal and $q=p^{e}$ for some $e \in \mathbb{N}$, and $I^{[q]}= \left( x^{q} \mid x \in I \right)$, the $q$-th Frobenius power of $I$. We denote the $\mathfrak{m}$-adic completion of $R$ by $\hat{R}$ and $\ell_{R}(-)$ denotes the length as an $R$-module. 

 The Hilbert-Kunz theory deals with the question of how $\ell_{R}\left(R/I^{[q]}\right)$ behaves as a function on $q$ and how understanding this behavior leads us to have a better understanding of the singularities of the ring. In his work, Kunz \cite{kunz1976noetherian} introduced this study to measure how close the ring $R$ is to be regular. Later P. Monsky \cite{monsky1983hilbert} showed that \\
\begin{equation*}
\label{Equation 1.1}
    e_{HK}(I,R)=\lim_{e \to \infty} \dfrac{\ell_{R}(R/I^{[q]})}{q^{d}} 
\end{equation*} exists for any $\mathfrak{m}$-primary ideal $I$, this positive real number is called the \textit{Hilbert-Kunz Multiplicity} of $I$. Contrary to the Hilbert-Samuel multiplicity of an $\mathfrak{m}$-primary ideal $I$, which is given by $ e(I,R) = \lim_{k\to \infty} \dfrac{\ell_{R}(R/I^{k})}{k^{d}/d!}$, the Hilbert-Kunz Multiplicity demonstrates much more complicated behavior \cite{han1993some}

In this paper, we work with certain families of ideals called $p$-families. A $p$-family of ideals is a sequence of ideals $ I_{\bullet}=\{I_{q}\}_{q=1}^{\infty}$ with $I_{q}^{[p]} \subseteq I_{pq}$ for all $q$. Given a $p$-family of $\mathfrak{m}$-primary ideals, we define the following limit as the \textit{volume} of the $p$-family.
\begin{equation*}
    \Vol_{R}(I_{\bullet})=\lim_{q\to \infty} \dfrac{\ell_{R}(R/I_{q})}{q^{d}}
\end{equation*}
In \cite{hernandez2018local}, Jack Jeffries and Daniel J Hern\'andez showed that $\Vol_{R}(I_{\bullet})$ exists as a limit if and only if $\Dim N(\hat{R}) <d$, where $N(\hat{R})$ denotes the nilradical of $\hat{R}$. This provides an alternating proof of the existence of Hilbert Kunz multiplicity and $F$-signature for a complete local domain (\cite{hernandez2018local}, Corollary 6.1,6.8).

This article aims to prove the following general \textit{Volume = Multiplicity} formula for $p$-families of $\mathfrak{m}$-primary ideals.\\ \\
\textbf{Theorem 5.3.} Let  $(R,\mathfrak{m},\mathbb{K})$ be a $d$-dimensional local ring of characteristic $p>0$. If $\Dim N(\hat{R}) < d$, then for any $p$-family of $\mathfrak{m}$-primary ideals $I_{\bullet}=\{I_{q}\}_{q=1}^{\infty}$\\
\begin{equation*}
   \Vol_{R}(I_{\bullet})= \lim_{q\to\infty} \frac{\ell_{R}(R/I_{q})}{q^{d}}=\lim_{q\to\infty} \frac{e_{HK}(I_{q},R)}{(q)^{d}}
\end{equation*}

To prove our main result, we have followed the approach of Cutkosky's existence theorem (\cite{cutkosky2013multiplicities} \cite{cutkosky2014asymptotic}), which can be outlined as follows:

\textit{(1)} Reduce to the case of a complete local domain.

\textit{(2)} Using a suitable valuation on this complete local domain $R$, we attach to every $p$-family of ideals a combinatorial structure, called $p$-system of ideals (see Definition \ref{P-system}) and associate an Euclidean object called $p$-body (see Definition \ref{P-body}).

\textit{(3)} To compute the relevant $R$-module lengths, we prove an approximation Theorem for $p$-systems of ideals. Since Lazarsfeld and Musta\c{t}\u{a} in their paper \cite[Proposition 3.1]{lazarsfeld2009convex} have proved similar approximation theorem and using it recovered the Fujita Approximation Theorem \cite[Theorem 3.3]{lazarsfeld2009convex}, we call it \textit{Fujita Type Approximation Theorem for $p$-systems of ideals} (see Theorem \ref{Fujita type approximation}).\\

Although the Hilbert-Kunz multiplicity behaves much differently compared to the Hilbert-Samuel multiplicity, a parallel analog of Theorem \ref{Volume =Multiplicity Formula} for every graded family $J_{\bullet}$ = $\{J_{n}\}_{n \in \mathbb{N}}$ (i.e., a sequence of ideals $J_{\bullet}$ such that $J_{m}. J_{n} \subseteq J_{m+n}$) of $\mathfrak{m}$-primary ideals and Hilbert-Samuel multiplicities has been proven for valuation ideals associated to an Abhyankar valuation in a regular local ring which is essentially of finite type over a field in \cite{ein2003uniform} when $R$ is a local domain which is essentially of finite type over an algebraically closed field $\mathbb{K}$ with $R/\mathfrak{m} = \mathbb{K}$ in \cite{lazarsfeld2009convex}, it is proven when $R$ is analytically unramified with perfect residue field in \cite{cutkosky2014asymptotic}. Finally in \cite{cutkosky2015general} Cutkosky settled this for any $d$-dimensional Noetherian local ring $R$ with $\Dim N(\hat{R}) < d$.\\

\section{Preliminaries }\label{Con&BasNot}
This section reviews a few basic notions from semigroup theory, convex geometry, and valuation theory.

Let $U$ and $V$ be arbitrary subsets of $\mathbb{R}^{d}$. If $U$ is Lebesgue measurable, its measure $\Vol_{\mathbb{R}^{d}}(U)$ is called the \textit{volume} of $U.$ We define the Minkowski sum of two sets as  \ $U+V := \left\{ \ \boldsymbol{u}+ \boldsymbol{v} \ | \  \boldsymbol{u} \in{U},   \boldsymbol{v} \in{V} \right\}.$

A \textit{convex cone} is any subset of $\mathbb{R}^{d}$ that is closed under taking an $\mathbb{R}$-linear combination of points with non-negative coefficients. Let $Cone(U) \subseteq \mathbb{R}^{d}$ be the convex cone which is the closure of the set of all linear combinations $\sum_{i} \lambda_{i}u_{i},$ with $u_{i} \in U$ and $\lambda_{i} \in \mathbb{R}_{\geqslant 0}$. Note that a cone in $\mathbb{R}^{d}$ has a non-empty interior if and only if the real vector space it generates has dimension $d.$ We call such a cone \textit{full-dimensional}. 

A cone $C$ is \textit{pointed} if it is closed and if there exists a vector $\boldsymbol{a} \in{\mathbb{R}^{d}}$ such that $ ( \boldsymbol{u}, \boldsymbol{a}) > 0$, for all $\boldsymbol{u} \in C \setminus \{\boldsymbol{0}\}$, where $ (-, -) $ is the inner product in $\mathbb{R}^{d}$. 

If $C$ is a pointed cone and $\alpha$ is a non-negative real number, we define:
\begin{equation*}
H = H_{\alpha} := \left\{ \ \boldsymbol{u} \in{\mathbb{R}^{d}} \ | \ (  \boldsymbol{u}, \boldsymbol{a}  ) < \alpha \ \right\}
\end{equation*}
a \textit{truncating half-space} for $C.$ 
\begin{remark}
\label{Remark(2.1)}
Let $C$ be a pointed cone. If $\alpha > 0,$ then a \textit{truncation} of $C$ is defined as $ C \cap H_{\alpha}$. It is a non-empty, bounded  subset of $C$.
\end{remark}
\begin{remark}
\label{Remark(2.2)}
 If $H$ is any truncating halfspace of a pointed cone $ C$ then so is $qH$, for all $q > 0$. 
\end{remark}
We define a \textit{semigroup} to be any subset of $\mathbb{Z}^{d}$ that contains $\boldsymbol{0}$ and it is closed under addition; e.g. $\mathbb{N}$ is the semigroup of non-negative integers. A semigroup $S$ is \textit{finitely generated} if there exist a finite subset $S_{0} \subseteq S$ such that every element of $S$ can be written as an $\mathbb{N}$-linear combination of elements of $S_{0}.$ A subset $T \subseteq S$ of a semigroup $S,$ is an \textit{ideal} of $S$ whenever $S+T$ is contained in $T.$ A semigroup $S$ is called \textit{pointed} if $\boldsymbol{a} \in S$ and $\boldsymbol{-a} \in S$ then $\boldsymbol{a}$ is $\boldsymbol{0}$.
\begin{remark}
\label{Remark(2.3)}
 If $\Cone (S)$ is pointed, then so is $S$, but the converse is generally false. However, it does hold if $S$ is assumed to be finitely generated.
\end{remark} 
\begin{remark}
\label{Remark(2.4)}
For any $\mathbb{Z}$-linear embedding $ i: \mathbb{Z}^{d} \hookrightarrow \mathbb{R}$, there exists some unique $\boldsymbol{a} \in \mathbb{R}^{d}$, whose coordinates are linearly independent over $\mathbb{Q}$, such that $i(\boldsymbol{u})=(\boldsymbol{a},\boldsymbol{u})$ for all $u \in \mathbb{Z}^{d}$.
We define $\boldsymbol{u} \leqslant_{\boldsymbol{a}} \boldsymbol{v}$ whenever $(\boldsymbol{a}, \boldsymbol{u}) \leqslant (\boldsymbol{a}, \boldsymbol{v})$.
\end{remark} 
Let $\mathbb{F}^{\times} = \mathbb{F}\setminus\left\{ 0  \right\}$. Fix, $\boldsymbol{a} \in{\mathbb{R}^{d}},$ and fix the embedding $\mathbb{Z}^{d} \hookrightarrow \mathbb{R}$ given by $ \boldsymbol{v} \to (\boldsymbol{a}, \boldsymbol{v}),$ defined by $\boldsymbol{a}$.  An $\boldsymbol{a}$-\textit{valuation} on a field $\mathbb{F}$ with value group $\mathbb{Z}^{d}$ is a surjective group homomorphism:
$\vartheta : \mathbb{F}^{\times} \rightarrow \mathbb{Z}^{d}$
with the property that
\begin{align*}
\vartheta(xy) = \vartheta(x)+\vartheta(y), \\
\vartheta(x+y) \geqslant_{\boldsymbol{a}} \mathrm{min} \left\{ \vartheta(x), \vartheta(y) \right\}.
\end{align*}
For any subset $N \subseteq \mathbb{F},$ with $N^{\times} = N \setminus \left\{ 0 \right\}$ we define $\vartheta(N) := \vartheta(N^{\times}),$ and this is called the \textit{image} of $N$ under $\vartheta.$ \\
Given a point $\boldsymbol{u} \in{\mathbb{Z}^{d}},$ we define:
\begin{align*}
\mathbb{F}_{\geqslant \boldsymbol{u}} := \left\{ \ x \in{\mathbb{F}} \ | \ \vartheta(x) \geqslant_{\boldsymbol{a}} \boldsymbol{u} \ \right\} \cup \left\{ 0 \right\} \quad \text{and} \quad
\mathbb{F}_{> \boldsymbol{u}} := \left\{ \ x \in{\mathbb{F}} \ | \ \vartheta(x) >_{\boldsymbol{a}} \boldsymbol{u} \ \right\} \cup \left\{ 0 \right\}
\end{align*}
The \textit{local ring} of $\vartheta$ is the subring $(V_{\vartheta}, m_{\vartheta}, k_{\vartheta})$ of $\mathbb{F},$ such that $V_{\vartheta} = \mathbb{F}_{\geqslant 0},$ and $m_{\vartheta} = \mathbb{F}_{> \boldsymbol{0}}.$ A local domain $(D, m, k)$ is \textit{dominated by} $(V_{\vartheta}, m_{\vartheta}, k_{\vartheta}),$ if $(D, m, k)$ is a local subring of $(V_{\vartheta}, m_{\vartheta}, k_{\vartheta})$, i.e., $\mathfrak{m} \subseteq \mathfrak{m}_{\vartheta}$.
\begin{remark}
\label{Remark(2.5)}
If $\boldsymbol{u} \in{\mathbb{Z}^{d}},$ then both $\mathbb{F}_{\geqslant \boldsymbol{u}}$ and $\mathbb{F}_{> \boldsymbol{u}}$ are modules over $V_{\vartheta}$ (as $V_{\vartheta} = \mathbb{F}_{\geqslant 0}$ and $\vartheta(xy) = \vartheta(x)+\vartheta(y)$) and $\mathbb{F}_{\geqslant \boldsymbol{u}}/\mathbb{F}_{> \boldsymbol{u}}$ is a vector space over $k_{\vartheta}.$ Note that, if $\boldsymbol{w} = \vartheta(f),$ \ for some non-zero $f \in{\mathbb{F}},$ then:
\begin{equation*}
\mathrm{dim}_{k_{\vartheta}} \left( \mathbb{F}_{\geq \boldsymbol{w}}/\mathbb{F}_{> \boldsymbol{w}} \right) = 1.
\end{equation*}
Indeed if $f$ is as above then for any $g$ such that $\vartheta(g) = \boldsymbol{w},$ we have $g = \left( \dfrac{g}{f} \right) \cdot f $, where $ \dfrac{g}{f} \in k_{\vartheta}$ because $ \vartheta \left( \dfrac{g}{f} \right )= 0$.
\end{remark}
\begin{remark}
\label{Remark(2.6)}
Let $(D,m,k)$ be a local domain of dimension $d$ with fractional field $\mathbb{F},$ and let $\vartheta : \mathbb{F}^{\times}\twoheadrightarrow \mathbb{Z}^{d}$ be an  $ \boldsymbol{a}$-\textit{valuation} with value group $\mathbb{Z}^{d}.$ Note that $S = \vartheta(D)$ is a subsemigroup of $\mathbb{Z}^{d},$ and the image of any ideal $I$ in $D$ is a semigroup ideal of $S.$ Moreover as $\vartheta$ is surjective, the Minkowski sum $S+(-S)$ is equal to $\mathbb{Z}^{d},$ as $\vartheta$ is surjective. Therefore, the real vector space generated by $\Cone(S)$ is $\mathbb{R}^{d}$, i.e., $\Cone(S)$ is full dimensional.

Observe that $D$ is dominated by $V_{\vartheta}$ if and only if for all $\boldsymbol{x} \in{\vartheta(m)},$ we have $\boldsymbol{x} >_{\boldsymbol{a}} \boldsymbol{0}$ and $S = \vartheta(m) \cup \left\{ \boldsymbol{0} \right\}$, i.e., any non-zero $\boldsymbol{u} \in{S},$ satisfies $(\boldsymbol{a}, \boldsymbol{u}) > 0.$
\end{remark}
\begin{definition}
\label{Strongly dominated}
Following Remark \ref{Remark(2.6)}, if $D$ is dominated by $V_{\vartheta}$ and $(\boldsymbol{a}, \boldsymbol{u}) > 0$, for every $\boldsymbol{u}  \in{\Cone(S) \setminus \left\{ \boldsymbol{0} \right\}}$, then we say $D$ is \textit{strongly dominated} by $\vartheta$.
\end{definition}

\section{OK-valuation, OK-domain, $p$-systems and $p$-bodies}\label{OK}
In this section, we review the constructions of  OK-valuation, OK-domain, $p$-systems, and $p$-bodies following the work of Cutkosky (\cite{cutkosky2013multiplicities},\cite{cutkosky2014asymptotic}) and Jeffries and Hern\'andez \cite{hernandez2018local}. The main result of this section is Theorem \ref{Fujita type approximation}.
\begin{definition}
\label{OK valuation}
Let $(D, m, k)$ be a $d$-dimensional local domain  with fractional field $\mathbb{F}$. Fix the embedding $\mathbb{Z}^{d} \hookrightarrow \mathbb{R}$  defined by $\boldsymbol{a} \in \mathbb{R}^{d}.$\\
If an $\boldsymbol{a}$-valuation
\begin{equation*}
     \vartheta : \mathbb{F}^{\times} \to \mathbb{Z}^{d}
\end{equation*}
with value group $\mathbb{Z}^{d},$ and local ring $(V_{\vartheta},m_{\vartheta},k_{\vartheta})$ satisfies the following conditions:\\
 \textit{i}. $D$ is strongly dominated by $V_{\vartheta}$, (see Definition \ref{Strongly dominated})\\
 \textit{ii}. the resulting extension of residue fields $k \hookrightarrow k_{\vartheta}$ is finite, and\\
 \textit{iii.} there exists a point  \textit{$\boldsymbol{v}$} $\in{\mathbb{Z}^{d}}$  such that:

\begin{equation*}
D \cap \mathbb{F}_{\geqslant n \boldsymbol{v}} \subseteq m^{n},  \ \forall n \in{\mathbb{N}}
\end{equation*}
then we say that $(V_{\vartheta},m_{\vartheta},k_{\vartheta})$ is \textit{OK-relative} to $D$.
\end{definition}
\begin{definition}
\label{OK Domain}
Let $D$ be a $d$-dimensional local domain with fraction field $\mathbb{F}$. If there exists a valuation $\vartheta$ on $\mathbb{F}$ with value group $\mathbb{Z}^{d}$, and valuation ring $(V_{\vartheta},m_{\vartheta},k_{\vartheta})$ that is \textit{OK-relative} to $D$, we say that $D$ is an \textit{OK-domain}.
\end{definition}
\begin{example}
\label{Example 3.3}
Let $R=\mathbb{K}\llbracket x_{1},...,x_{d} \rrbracket$  is a power series ring  of dimension $d$, containing a field $\mathbb{K}$, $\mathfrak{m}=(x_{1},...,x_{d})$ is the maximal ideal and $\mathbb{F}$ is the quotient field of $R$. Take $a_{1},...,a_{d}$ be rationally independent real numbers and without loss of generality assume $a_{i} >0$ for all $1\leqslant i \leqslant d $. Define an ordering on $\mathbb{Z}^{d}$ by $\boldsymbol{a}=(a_{1},...,a_{d})$.\\
Let $f \in R$; then we can write, \\
$$f= \sum_{ s_{i_{1},...,i_{d}} \in \mathbb{K} , (i_1,...,i_d) \in \mathbb{N}^{d}} s_{i_{1},...,i_{d}}x_{1}^{i_{1}}...x_{d}^{i_{d}}. $$\\
Since $\boldsymbol{a}$ has positive entries, this will imply that every subset of $\mathbb{N}^{d}$ has a least element, i.e., there exist a well defined least exponent vector $(i_{1}^{'},...,i_{d}^{'})$ among its terms and we define, $\vartheta(f)=(i_{1}^{'},...,i_{d}^{'})$.\\
We claim that the induced valuation $\vartheta : \mathbb{F}^{\times} \to \mathbb{Z}^{d}$ is OK relative to $R$.\\
Let $S=\vartheta(R)$, therefore from the above definition it is clear that $C=\Cone(S)$ is contained in the non-negative octant, then $(\boldsymbol{u},\boldsymbol{a}) > 0$ for every nonzero $\boldsymbol{u}$ in $C$.
This shows that $R$ is strongly dominated by $\vartheta$, satisfying condition \textit{i}.
Let $(V,\mathfrak{m}_{\vartheta},\mathbb{K}_{\vartheta})$ is the valuation ring $\mathbb{F}_{\geqslant \boldsymbol{0}}$ with it's maximal ideal $\mathbb{F}_{ > \boldsymbol{0}}$, then clearly $\mathbb{K}_{\vartheta}= V/\mathfrak{m}_{\vartheta} \cong R/\mathfrak{m} \cong \mathbb{K}$, which satisfies condition \textit{ii}.
Moreover, let $a_{j}=\Max\{a_{1},...,a_{d}\}$, then we claim that $\vartheta(x_{j})=(0,..,1,0,..,0)=\boldsymbol{v}$ is the vector for condition \textit{iii}. Indeed let,
\begin{flalign*}
    f \in R \cap \mathbb{F}_{\geqslant a\boldsymbol{v}} \implies \vartheta(f) \geqslant a \boldsymbol{v} \implies (i_{1},...,i_{d}) \geqslant (0,..,a,0,..,0)
    \end{flalign*}
    \begin{flalign*}
        \implies a_{1}i_{1}+...+a_{d}i_{d} \geqslant aa_{j} \implies aa_{j} \leqslant a_{j}i_{1}+...+a_{j}i_{d}   (\because a_{j}=\Max\{a_{1},...,a_{d}\})
         \end{flalign*}
         \begin{flalign*}
            \implies a \leqslant i_{1}+...+i_{d} \implies R \cap \mathbb{F}_{\geqslant a\boldsymbol{v}} \subseteq \mathfrak{m}^{a}
\end{flalign*}
     
This finishes the proof that every power series ring containing a field is an OK-domain.

\end{example}
\begin{example}  \cite [Example 3.5]{hernandez2018local})
\label{Example 3.4}
This immediately implies a regular local ring $R$ containing a field is an OK-domain.
\end{example}
 \begin{example}
     
  \cite [Corollary 3.7]{hernandez2018local}
 \label{Example 3.5}
An interesting class of OK-domain examples is excellent local domains containing a field. In particular, a complete local domain containing a field is an OK-domain.
\end{example}

For the rest of this section, we fix a $d$-dimensional local domain $D$ of characteristic $p > 0$ with residue field $k$ and fraction field $\mathbb{F},$ a $\mathbb{Z}$-linear embedding of $\mathbb{Z}^{d}$ into $\mathbb{R}$ induced by a vector $\boldsymbol{a}$ in $\mathbb{R}^{d}$, and a valuation $\vartheta : \mathbb{F}^{\times} \twoheadrightarrow \mathbb{Z}^{d}$ that is OK relative to $D$. We use $S$ to denote the semigroup $\vartheta(D)$ in $\mathbb{Z}^{d}$, and $C$ to denote the closed cone in $\mathbb{R}^{d}$ generated by $S$. We follow the notation established in Definition \ref{OK valuation}.
\begin{remark}
\label{Remark 3.6}
If $M$ is a $D$-submodule of $\mathbb{F},$ then the $D$-module structure on $M$ induces a $k$-vector space structure on the quotient $\dfrac{M \cap \mathbb{F}_{\geqslant \boldsymbol{u}}}{M \cap \mathbb{F}_{> \boldsymbol{u}}}$ and by definition this space is non-zero if and only if there exists an element $x \in{M}$ with $\vartheta(x) = \boldsymbol{u}.$
\end{remark}
\begin{definition}
\label{Definition 3.7}
For a $D$-submodule $M$ of $\mathbb{F}$, we define:
\begin{equation*}
\vartheta^{(h)}(M)=\left \{ \boldsymbol{u} \in \mathbb{Z}^{d} \mid \mathrm{dim}_{k} \left( \dfrac{M \cap \mathbb{F}_{\geqslant \boldsymbol{u}}}{M \cap \mathbb{F}_{> \boldsymbol{u}}} \right) \geqslant h \right\},
\end{equation*}
where $1 \leqslant h \leqslant [k_{\vartheta}:k]$.
\end{definition}
\begin{remark}
\label{Remark 3.8}
Let $g \in D^{\times},$ with $\boldsymbol{v} = \vartheta(g),$  and let $\boldsymbol{u} \in \mathbb{Z}^{d}.$ Let $I_{\bullet}=\{I_{q}\}_{q=1}^{\infty}$ be a sequence of ideals indexed by $q=p^{e}$, then the map:
\begin{equation*}
\dfrac{I_{q} \cap \mathbb{F}_{\geqslant \boldsymbol{u}}}{I_{q} \cap \mathbb{F}_{> \boldsymbol{u}}} \rightarrow \dfrac{I_{q} \cap \mathbb{F}_{\geqslant \boldsymbol{u}+\boldsymbol{v}}}{I_{q} \cap \mathbb{F}_{> \boldsymbol{u}+\boldsymbol{v}}} \ : \ [x] \mapsto [gx]
\end{equation*}
is a $k$-linear injection for all $\boldsymbol{u} \in{\mathbb{Z}^{d}}$ and for all  $I_{q}$ in that sequence. Therefore using Definition \ref{Definition 3.7}, we have $\vartheta^{(h)}(I_{q}) + S \subseteq \vartheta^{(h)}(I_{q}) $ ,i.e., $\vartheta^{(h)}(I_{q})$ is an ideal of $S=\vartheta(D)$ for all $1 \leqslant h \leqslant [k_{\vartheta}:k]$.
\end{remark}
\begin{lemma}\cite[Lemma 3.11]{hernandez2018local}
\label{Lemma 3.9}
For a $D$-submodule $M$ of $\mathbb{F}$ and $\boldsymbol{v} \in{\mathbb{Z}^{d}},$ we have:
\begin{equation*}
\ell_{D} \left( M/(M \cap \mathbb{F}_{\geqslant \boldsymbol{v}}) \right) = \sum_{h = 1}^{[k_{\vartheta} : k]}  \#(\vartheta^{(h)}(M) \cap H)
\end{equation*}
where $H= \left\{\boldsymbol{u} \in{\mathbb{R}^{d}} \ | (  \boldsymbol{a}, \boldsymbol{u} ) < ( \boldsymbol{a}, \boldsymbol{v} )  \right\}.$
\end{lemma}
 We now discuss the notions of \textit{$p$-system} and \textit{$p$-bodies}.
\begin{definition}
\label{3.10}
A semigroup $S$ is called \textit{standard} if $S-S= \mathbb{Z}^{d}$, and the full dimensional cone generated by $S$ in $\mathbb{R}^{d}$ is pointed. Following the discussion in Remark \ref{Remark(2.6)} and from Definition \ref{OK valuation} the main example of a standard semigroup in $\mathbb{Z}^{d}$ is $\vartheta(D) $ associated to an $\boldsymbol{a}$-valuation $\vartheta$ that is \textit{OK} relative to some $d$-dimensional local domain $D$.
\begin{definition}
\label{P-system}
A collection of subsets $T_{\bullet}=\{T_{q}\}_{q=1} ^{\infty}$ of a semigroup $S$ indexed by $q=p^{e}$ satisfying\\
\textit{i.} $T_{q}$ is an ideal of S. ($T_{q} + S \subseteq S) $ and 
\textit{ii.} $pT_{q} \subseteq T_{pq}$
is called a \textit{$p$-system}.
\end{definition}
\begin{definition}
\label{P-body}
The $p$-body associated to a given $p$-system of ideals $T_{\bullet}$ of a  semigroup in $\mathbb{Z}^{d}$  is defined by \\
$\Delta( S, T_{\bullet})=\bigcup_{q=1}^{\infty} \Delta_{q}(S,T_{\bullet}) $  where
$\Delta_{q}(S,T_{\bullet})= \frac{1}{q}T_{q} + \Cone(S)$
\end{definition}
\begin{remark}
\label{Remark 3.13}
Although we have defined the notion of $p$-body for a $p$-system of ideals in any semigroup, in our situation, we are mostly concerned with $p$-system of ideals in the standard semigroup. Moreover, let $q_{1}=p^{e_{1}}$ and $ q_{2}=p^{e_{2}}$ and $e_{1} \leqslant e_{2}$, then
$p^{e_{2}-e_{1}} T_{p^{e_{1}}} \subseteq T_{p^{e_{2}}}$ (using Definition \ref{P-body}) this implies $\frac{T_{{p}^{e_{1}}}}{p^{e_{1}}} \subseteq \frac{T_{{p}^{e_{2}}}}{p^{e_{2}}} $. Therefore the $p$-body associated to a given $p$-system of ideals is an ascending union of sets.
\end{remark}
\end{definition}
\begin{example} \cite[Example 4.5]{hernandez2018local}
\label{Example 3.14}
Let $T$ be any subset of $S$, define $T_{q}=T+S$, for all $q=p^{e}$, then:\\
\textit{i.} Clearly $T_{q}+S \subseteq T_{q}$\\
\textit{ii.} $p T_{q}=(p T +p S) \subseteq (p T + S) \subseteq (T+ (p-1) T +S) \subseteq T+S $ which is equal to $T_{pq}$.\\
Therefore $T_{q} $ is a $p$-system of ideals for all $q=p^{e}$. Moreover, we have $\Delta_{q}(S, T_{\bullet})=\left(\frac{1}{q} T +\Cone(S)\right)$ and the closure of $\Delta(S,T_{\bullet})$ equals $\Cone(S)$.
\end{example}
\begin{example}
\label{Example 3.15}
Let $D$ be an OK domain for an OK valuation $\vartheta$ and assume $S=\vartheta(D)$ is the corresponding semigroup. Take an ideal $I$ in $D$ and define $T_{q}=q \vartheta(I) + S$. Following the same argument as in Example \ref{Example 3.14}, one can show that $\{T_{q}\}_{q=1}^{\infty}$  is a $p$-system of ideals. Moreover $\Delta(S,T_{\bullet})=\Delta_{q}(S,T_{\bullet})= \vartheta(I)+\Cone(S)$.
\end{example}

In the previous example, one can replace $\vartheta(I)$ with any arbitrary subset $T$ of $S$, and the corresponding $p$-body would be $T+\Cone(S)$. This shows that $p$ bodies need not be convex, and in practical applications, mostly, they are not.

\begin{remark}
\label{Remark 3.16}
Following Definition \ref{P-body} one can see that every $p$-body $\Delta$ is the union of countably many translates of $C=\Cone(S)$ therefore, it is Lebesgue measurable. If $H$ is a truncating halfspace of $C$, then the volume of ($\Delta \cap H$) is a well-defined real number.
\end{remark}
The following Theorem describes  $\Vol_{\mathbb{R}^{d}}(\Delta \cap H)$ for any truncating halfspace $H$ of $\Cone(S)$. This Theorem is one of the key ingredients in the proof of our main Theorem.
\begin{theorem} \cite[Theorem 4.10]{hernandez2018local}
\label{Theorem 3.17}
For a standard semigroup $S$ in $\mathbb{Z}^{d}$, a $p$-system $T_{\bullet}$ in $S$, and a truncating halfspace $H$  for $\Cone(S)$, we have:\\
\begin{equation*}
\label{Equation(3.2)}
  \lim_{q\to\infty}  \frac {\# (T_{q} \cap qH)}{q^{d}}= \Vol_{\mathbb{R}^{d}}(\Delta(S,T_{\bullet})\cap H)
\end{equation*}
\end{theorem}

Lazarsfeld and Musta\c{t}\u{a} in their paper \cite[Proposition 3.1]{lazarsfeld2009convex} have proved an approximation theorem similar to the following Theorem and using this in a global setup with $D$ a big divisor on an irreducible projective variety; they have recovered the Fujita Approximation Theorem \cite[Theorem 3.3]{lazarsfeld2009convex}. Since our approximation is similar to a different set of properties on the semigroup, we call it \textit{Fujita Type Approximation Theorem for $p$ system of ideals}.
\begin{theorem}
\label{Fujita type approximation}
Let S be a standard semigroup in $\mathbb{Z}^{d}$. If $T_{\bullet}$ is a $p$-system of ideals in $S$, and let $H $ is any truncating halfspace for $\Cone(S)$, then for any given $\epsilon > 0$, there exists $q_{0}$ such that if $q\geqslant q_{0}$ the following inequality hold. 
\begin{equation*}
  \lim_{e\to\infty} \frac {\# ((p^{e}T_{q}+S) \cap p^{e}qH)}{p^{ed}q^{d}}\geqslant \Vol_{\mathbb{R}^{d}}(\Delta(S,T_{\bullet})\cap H)-\epsilon
\end{equation*}
\end{theorem}
\begin{proof}
By Example \ref{Example 3.15} we notice that $T_{\bullet}^{'} = \{p^{e}T_{q}+S\}_{e=1}^{\infty}$ is a $p$-system of ideals. Therefore using Theorem \ref{Theorem 3.17} and Remark \ref{Remark(2.2)}  we have, \\
\begin{equation}
\label{Equation 3.1}
    \lim_{e\to\infty} \frac {\# ((p^{e}T_{q}+S) \cap p^{e}qH)}{p^{ed}}= \Vol_{\mathbb{R}^{d}}((T_{q}+\Cone(S))  \cap qH)
\end{equation}
Now note that $ \frac{\Vol_{\mathbb{R}^{d}}((T_{q}+\Cone(S) \cap qH)}{q^{d}} = \Vol_{\mathbb{R}^{d}}((\frac{T_{q}}{q} +\Cone(S)) \cap H) $ and from Remark \ref{Remark 3.13} and Remark \ref{Remark 3.16}, we know that these are an ascending union of measurable sets. Therefore,
\begin{equation}
\label{equation 3.2}
     \lim_{q\to\infty} \frac {\Vol_{\mathbb{R}^{d}}((T_{q}+\Cone(S)) \cap qH)}{q^{d}}= \Vol_{\mathbb{R}^{d}}(\Delta(S,T_{\bullet})  \cap H)
\end{equation}

We get Equation \ref{equation 3.2} by using the sub-additivity and continuity from below properties of the Lebesgue measure. Now combining Equation \ref{Equation 3.1}, Equation \ref{equation 3.2} for a chosen $\epsilon > 0$, we can find $q_{0}$ such that for all $q\geqslant q_{0}$ \\
$\lim_{e\to\infty} \frac {\# ((p^{e}T_{q}+S) \cap p^{e}qH)}{p^{ed}q^{d}}\geqslant \Vol_{\mathbb{R}^{d}}(\Delta(S,T_{\bullet})\cap H)-\epsilon$
\end{proof}

\section{p-families and a brief introduction to Hilbert Kunz Multiplicity}\label{Section 4}

In this section, all rings will be commutative, Noetherian, and of prime characteristic, $p > 0$.

Let $I$ be an ideal of a ring $R$, generated by $\{x_{1},...,x_{r}\}$ then \textit{$p^{e}$}th Frobenius power of $I$ denoted by $I^{[p^{e}]} $ is generated by $\{x_{1}^{p^{e}},..., x_{r}^{p^{e}}\}$.
\begin{definition}
\label{ P-family}
A sequence of ideals $I_{\bullet}=\{ I_{p^e}\}_{e=0}^{\infty}$ is called a $p$-family whenever $I_{q}^{[p]} \subseteq I_{pq}$, where $q=p^e$ for some $e \geqslant 0$.
\end{definition}
\begin{remark}
\label{Remark 4.2}
We note that if $(R,\mathfrak{m})$ is local, then every term in this family is $\mathfrak{m}$-primary if and only if $I_{1}$ is $\mathfrak{m}$ primary. 
\end{remark}
\begin{example}
\label{Example 4.3}
Let $I_{\bullet}=\{J^{[p^{e}]}\}_{e=0}^{\infty}$ for some fixed ideal $J$, then $I_{p^{e}}^{[p]}=J^{[p^{e+1}]}=I_{p^{e+1}} $, therefore $I_{\bullet}$ is a $p$-family.
\end{example}
\begin{lemma}
\label{Lemma 4.4}
Let $(R,\mathfrak{m})$ be a local ring and let $I_{\bullet}=\{I_{q}\}_{q=1}^{\infty}$ be a $p$-family of $\mathfrak{m}$-primary ideals. Then there exist a $c>0$ such that $\mathfrak{m}^{cq} \subseteq I_{q}$ and $\mathfrak{m}^{cp^{e}q} \subseteq I_{q}^{[p^{e}]}$ for all $q$ a power of $p$ and for all $e\geqslant 0$.
\end{lemma}
\begin{proof}
We know from Remark \ref{Remark 4.2} that there exist $c_{1}$ such that $(\mathfrak{m}^{c_{1}})^{[q]} \subseteq I_{q}$ for all $q$, a power of $p$. Let $\mathfrak{m}$ be generated by $b$ elements then using pigeon-hole principle, $\mathfrak{m}^{bc_{1}p^{e}q} \subseteq (\mathfrak{m}^{c_{1}})^{[p^{e}q]} \subseteq ((\mathfrak{m}^{c_{1}})^{[q]})^{[p^{e}]} \subseteq I_{q}^{[p^e]}$, for all $q$ a power of $p$ and for all $e\geqslant 0$. So, letting $c=bc_{1}$ we get our desired conclusion.
\end{proof}
 We define the Frobenius map $F: R \to R$ by $F(r)=r^{p}$. This map turns $R$ into an $R$-module with a nonstandard action $r.x=r^{p}x$, and we call this module $F_{*}R$. Inductively one can define $F_{*}^{e} R $. We define an $R$-linear map $\phi_{e}: F_{*}^{e}R \to R$ which is a set map from $R$ to $R$ such that $\phi_{e}$ is additive and $\phi_{e}(x^{p^e}y)=x \phi(y)$.\\
The following is an interesting example of $p$-families.
\begin{example}
\label{Example 4.5}
Assume $R$ is reduced and $I$ is an $R$-ideal. Define for all $q=p^{e}$, $J_{q}=\{x \in R \mid  \phi(x) \in I,   \forall  \phi \in Hom_{R}(F_{*}^{e}(R), R)\}$. Notice that $J_{q}$ is an ideal for each $q=p^{e}$. Since $R$ is reduced we can identify $F_{*}^{e}(R)$ as $R^{\frac{1}{p^{e}}}$, the ring of $p^{e}$-th roots of the elements in $R$. We claim that this is a $p$-family of ideals. Indeed, let $\phi \in Hom_{R}(F_{*}^{e+1}R,R)$ and choose $r \in J_{q}$ then $\psi(r^{\frac{1}{p^{e}}}) \in I$ for all $ \psi \in Hom_{R}(F_{*}^{e}R,R)$ . Now $\phi\left((r^{p})^{\frac{1}{p^{e+1}}}\right)=\phi_{e}\left(r^{\frac{1}{p^{e}}}\right) \in I$, as $r \in J_{q}$. This implies $J_{q}^{[p]} \subseteq J_{pq}$.
\end{example}
\begin{remark}
\label{Remark 4.6}
If in the previous example we choose $I=\mathfrak{m}$, then we get a $p$-family of $m$-primary ideals because $\mathfrak{m}^{[q]} \subseteq J_{q}$. Use of this sequence of ideals was present in the work of Y. Yao \cite{yao2006observations} as well as F. Enescu and I. Aberbach \cite{aberbach2005structure}. Later this sequence of ideals was used in the work of K. Tucker \cite{tucker2012f} for proving that $F$-signature exists.

Let $(D,m,k)$ be a $d$-dimensional local OK-domain with OK-valuation $\vartheta$. Choose any $p$ family $I_{\bullet}$ of ideals, then \\
\textit{i.} $\vartheta(I_{q})$ is an ideal of the semigroup $\vartheta(D)=S$, and, \\
\textit{ii.} $p\vartheta(I_{q}) \subseteq \vartheta(I_{q}^{[p]}) \subseteq \vartheta(I_{pq})$. Therefore, $\{\vartheta(I_{q})\}_{q=1}^{\infty}$ is a $p$-system of ideals in $S$.
\end{remark}
\begin{remark}
\label{Remark(4.7)}
Following Definition \ref{Definition 3.7} and Remark \ref{Remark 3.8}, one can observe that although $\vartheta^{(h)}(I_{q})$ is an ideal but we need the extra condition that $\{m_{1},...,m_{h}\}$ are $k$-linear independent in $\vartheta^{(h)}(I_{q})$ implies $\{m_{1}^{p},...,m_{h}^{p}\}$ are $k$-linear independent in $\vartheta^{(h)}(I_{pq})$ but it doesn't always true unless we assume that $k$ is perfect. In (\cite[Lemma 3.12]{hernandez2018local}) they get around it by uniformly approximating $\vartheta^{h}(I_{q})$ with $\vartheta(I_{q})$ and proving the following result:
\end{remark}
\begin{proposition} \cite[Corollary 5.10]{hernandez2018local}
\label{Proposition 4.8}
 Let $D$ be a $d$-dimensional local OK-domain with OK-valuation $\vartheta$. For a $p$-family of ideals $I_{\bullet}$ in $D$ we have:
\begin{equation*}
\label{Equation(4.1)}
    \lim_{q\to\infty} \frac {(\vartheta^{(h)}(I_{q)} \cap qH)}{q^{d}}= \Vol_{\mathbb{R}^{d}}(\Delta(S,\vartheta (I_{\bullet}))\cap H)
\end{equation*}
where $S=\vartheta(D)$, $C= \Cone (S)$ and $H$ is any truncating halfspace of $C$.
\end{proposition}

We will now discuss Hilbert-Kunz Multiplicity and some important properties.\\

For the rest of this section let $(R,\mathfrak{m},\mathbb{K})$ be a $d$-dimensional local ring and $q=p^{e}$, for some $e \in \mathbb{N}$.
\begin{definition}
\label{Hilbert-Kunz Multiplicity}
Let $I$ be an $\mathfrak{m}$-primary ideal of $R$. We define Hilbert-Kunz Multiplicity of $I$ by $e_{HK}(I,R)=\lim_{q \to \infty} \dfrac{\ell \left( R/I^{[q]}\right)}{q^{d}}$ 
\end{definition}
\begin{remark}
\label{Remark 4.10}
Let $J$ be an $\mathfrak{m}$-primary ideal of $R$, then  $\ell \left(R/J \right)$ is unaffected by completion. Therefore using the above definition, we have
\begin{equation*}
\label{Equation(4.2)}
    e_{HK}(J,R)=e_{HK}(J \hat{R},\hat{R})
\end{equation*}
\end{remark}
\begin{lemma}
\label{Lemma 4.11}
Let $I$ be an $\mathfrak{m}$-primary ideal of $R$ and $M$ be a finitely generated $R$-module. Then there exists a constant $\alpha > 0$ such that
\begin{equation*}
\label{Equation(4.3)}
    \ell_{R}\left(\frac{M}{I^{[p^{e}]}M}\right) \leqslant \alpha p^{e.\Dim M}
\end{equation*}
\end{lemma}
\begin{proof}
Since $I$ is $\mathfrak{m}$-primary there exist $c >0$ such that $\mathfrak{m}^{c} \subseteq I $. Let $b$ is the minimal number of generators of $I$, then $\mathfrak{m}^{p^{e}bc} \subseteq I^{p^{e}b} \subseteq I^{[p^{e}]}$ (using pigeon-hole principle). Therefore $\ell_{R}\left(\frac{M}{I^{[p^{e}]}}\right) \leqslant \ell_{R}\left(\frac{M}{\mathfrak{m}^{p^{e}bc}}\right)$ and the later agrees with a polynomial in $p^{e}bc$ of degree $\Dim M $. If the leading coefficient of this polynomial is $a$ then choose $a_{0} >> a$ implies $\ell_{R}(\frac{M}{\mathfrak{m}^{p^{e}bc}})$ is bounded above by , $a_{0}(b^{\Dim M}c^{\Dim M}p^{e \Dim M})$ taking $\alpha=a_{0}b^{\Dim M}c^{\Dim M}$ will finish the proof.
\end{proof}
\begin{corollary}
\label{Corollary 4.12}
Let $N$ be an ideal of $R$ and $A=R/N$, then for any $\mathfrak {m}$-primary ideal $I$ of $R$ there exist $\beta > 0$ such that \\
\begin{equation*}
   0\leqslant  \ell_{R}\left(R/I^{[p^{e}]}\right) -\ell_{A}\left(A/I^{[p^{e}]}\right) \leqslant \beta . {p^{e.\Dim{N}}}
\end{equation*}
\end{corollary}
\begin{proof}
Take the following exact sequence \\
\begin{equation*}
\label{Equation(4.4)}
    0 \to N/(N \cap I^{[p^{e}]}) \to R/{I^{[p^{e}]}} \to A/{I^{[p^{e}]}A} \to 0 
\end{equation*}
Using the length formula, we have \\
\begin{equation}
\label{Equation(4.5)}
    \ell_{R}\left(R/{I^{[p^{e}]}}\right)-\ell_{A}\left(A/{I^{[p^{e}]}}\right)=\ell_{R}\left(N/(N \cap I^{[p^{e}]})\right)
\end{equation}
Note that $\ell_{R}\left(A/{I^{[p^{e}]}}\right) = \ell_{A}\left(A/{I^{[p^{e}]}}\right)$ and $I^{[p^{e}]}N \subseteq I^{[p^{e}]} \cap N$. Now using Lemma \ref{Lemma 4.11}, we have from Equation \ref{Equation(4.5)},       $\ell_{R}\left(R/{I^{[p^{e}]}}\right)-\ell_{A}\left(A/{I^{[p^{e}]}}\right) \leqslant \ell_{R}\left(N/{I^{[p^{e}]}N}\right) \leqslant \beta p^{e.\Dim N}$
\end{proof}
\begin{definition}
\label{Definition 4.13}
Let $f,g : \mathbb{N} \to \mathbb{R}$ be real valued functions from the set of non-negative integers. We say $f(n)=O(g(n))$ if there exists a positive constant $C$ such that $|f(n)| \leqslant C g(n)$ for all $n \gg 0$ and we say $f(n)=o(g(n))$ if $ \lim_{n \to \infty} \dfrac{f(n)}{g(n)}=0$
\end{definition}
Now we will mention some results on Hilbert Kunz Multiplicity without proof; for a detailed verification, we refer readers to \cite{huneke2013hilbert}.
\begin{proposition}
\label{Proposition 4.14}
Let \\
\begin{equation*}
    0 \to N \to M \to K \to 0
\end{equation*}
be a short exact sequence of finitely generated $R$-modules. Then \\
\begin{equation}
\label{Equation(4.6)}
    \ell_{R}\left(M/{I^{[q]}M}\right)=\ell_{R}\left(N/{I^{[q]}N}\right)+\ell_{R}\left(K/{I^{[q]}K}\right)+O(q^{d-1})
\end{equation}
\end{proposition}
Now dividing Equation \ref{Equation(4.6)} by $q^{d}$ and taking $q\to \infty$ we have\\
\begin{equation*}
\label{Equation(4.7)}
    e_{HK}(I,M)=e_{HK}(I,K)+e_{HK}(I,N)
\end{equation*}
\begin{theorem}
\label{Theorem 4.15}
Let $I$ be an $\mathfrak{m}$-primary ideal of $R$ and $M$ is a finitely generated $R$-module. Let $\Gamma$ be the set of minimal prime ideals $P$ of $R$ such that $ \Dim \left(R/P\right)=\Dim (R)$. Then \\
\begin{equation*}
\label{Equation(4.8)}
    e_{HK}(I,M)= \sum_{p \in \Gamma} e_{HK}\left(I,R/P\right) \ell(M_{P})
\end{equation*}
\end{theorem}
 The main idea to prove this result is to take a prime filtration of $M$, use the fact that $e_{HK}(I, R/Q)=0$ if $\Dim R/Q < \Dim R$ and then use Proposition \ref{Proposition 4.14}.
\begin{remark}
\label{Remark(4.16)}
In the Theorem \ref{Theorem 4.15}, if we choose $M=R$ and further assume  $R$ is reduced. Then
\begin{equation*}
\label{Equation(4.9)}
    e_{HK}(I,R)= \sum_{p \in \Gamma} e_{HK}\left(I,R/P\right)
\end{equation*}
because a reduced local ring with Krull dimension $0$ is a field, therefore $\ell(R_{P})=1$
\end{remark}

\section{Volume=Multiplicity Formula } \label{Section 5}
In this section we prove our main result, a general \textit{Volume=Multiplicity} formula for $p$-families of $\mathfrak{m}$-primary ideals. We begin with the following important  lemma.
\begin{lemma} \cite[Lemma 5.21]{hernandez2018local}
\label{Lemma 5.1}
Let $(R,\mathfrak{m},\mathbb{K})$ be a $d$-dimensional reduced local ring of positive characteristic, and $I_{\bullet}$ be a sequence of ideals of $R$ indexed by the powers of $p$ such that $\mathfrak{m}^{cq}\subseteq I_{q}$ for some positive integer $c$ and for all $q=p^{e}, e \in \mathbb{N}$. Let $P_{1},...,P_{n}$ be the minimal primes of $R$, then there exists $\delta >0$ such that for all $q$,
\begin{equation*}
    \left|\sum_{i=1}^{n} \ell_{R_{i}} \left(R_{i}/{I_{q}R_{i}}\right)-\ell_{R}\left(R/{I_{q}}\right)\right| \leqslant \delta.q^{d-1}
\end{equation*}
where $R_{i}=R/P_{i}$, for all $i \in \{1,...,n\}$
\end{lemma}
\begin{remark}
\label{Remark 5.2}
Suppose $(R,\mathfrak{m})$ be a $d$-dimensional reduced local ring of chracteristic $p>0$ and $I$ is an $\mathfrak{m}$-primary ideal. Since $R$ is reduced and $I_{\bullet}=\{I^{[p^{e}]}\}_{e=0}^{\infty}$ is a $p$-family of $\mathfrak{m}$ primary ideals, using Lemma \ref{Lemma 5.1}, Lemma \ref{Lemma 4.4} one can give an alternative proof of Remark \ref{Remark(4.16)}.
\end{remark}
The following is the Main Theorem of this paper.
\begin{theorem}
\label{Volume =Multiplicity Formula}
Let  $(R,\mathfrak{m},\mathbb{K})$ be a $d$-dimensional local ring of characteristic $p>0$. If the  $R$-module dimension of the nilradical of $\hat{R}$ is less than $d$, then for any $p$-family of $\mathfrak{m}$-primary ideals $I_{\bullet}=\{I_{q}\}_{q=1}^{\infty}$, we have
\begin{equation*}
   \lim_{q\to\infty} \frac{\ell_{R}(R/I_{q})}{q^{d}}=\lim_{q\to\infty} \frac{e_{HK}(I_{q},R)}{(q)^{d}}
\end{equation*}
\end{theorem}
\begin{proof}

\textbf{Step 1: The case of OK-Domain:}

Assume $R$ is an OK-domain and let $\vartheta:\mathbb{F}^{\times} \to \mathbb{Z}^{d}$ be the valuation that is OK-relative to $R$ with respect to a $\mathbb{Z}$-linear embedding of $\mathbb{Z}^{d}$ into $\mathbb{R}$ induced by a vector $\boldsymbol{a}$. Let $(V,\mathfrak{m}_{\vartheta},\mathbb{K}_{\vartheta})$ be the associated valuation ring, $S=\vartheta(R)$ be the corresponding semigroup in $ \mathbb{Z}^{d}$ and $C= \Cone (S)$.\\
Let $\boldsymbol{u} \in \mathbb{Z}^{d}$; then we have the following two exact sequences:
\begin{equation}
\label{Equation(5.3)}
    0 \to \frac{I_{q}}{I_{q} \cap \mathbb{F}_{\geq \boldsymbol{u}}} \to \frac{R}{I_{q} \cap \mathbb{F}_{\geq \boldsymbol{u}}} \to \frac{R}{I_{q}} \to 0
\end{equation}
\begin{equation}
\label{Equation(5.4)}
    0 \to \frac{R \cap \mathbb{F}_{\geqslant \boldsymbol{u}}}{I_{q} \cap \mathbb{F}_{\geqslant \boldsymbol{u}}} \to \frac{R}{I_{q} \cap \mathbb{F}_{\geqslant \boldsymbol{u}}} \to \frac{R}{R \cap \mathbb{F}_{\geqslant \boldsymbol{u}} } \to 0
\end{equation}
From Equation \ref{Equation(5.3)}, \ref{Equation(5.4)}, we have
\begin{equation}
\label{Equation(5.5)}
    \ell_{R}\left(R/{I_{q}}\right)= \ell_{R}\left(\frac{R  \cap  \mathbb{F}_{\geqslant \boldsymbol{u}}}{I_{q} \cap \mathbb{F}_{\geqslant \boldsymbol{u}}}\right) + \ell_{R}\left(\frac{R}{R \cap \mathbb{F}_{\geqslant \boldsymbol{u}} }\right) - \ell_{R}\left( \frac{I_{q}}{I_{q} \cap \mathbb{F}_{\geqslant \boldsymbol{u}}}\right)
\end{equation}
Let ${\boldsymbol{v}} \in S$ satisfy the last condition in Definition \ref{OK valuation}. Choose the number $c>0$ from Lemma \ref{Lemma 4.4}, define $\boldsymbol{w}$=$c\boldsymbol{v}$, then $R \cap \mathbb{F}_{\geqslant q\boldsymbol{w}} \subseteq \mathfrak{m}^{cq} \subseteq I_{q}$. Therefore\\
\begin{equation}
\label{Equation(5.6)}
    I_{q} \cap \mathbb{F}_{\geqslant q\boldsymbol{w}} = R \cap \mathbb{F}_{\geqslant q\boldsymbol{w}}
\end{equation}
Let $H$ be the halfspace $\{\boldsymbol{u}^{'}  \ |  (\boldsymbol{a}, \boldsymbol{u}^{'}) <   (\boldsymbol{a}, \boldsymbol{w}  ) \}$. Then using Equation \ref{Equation(5.5)}, \ref{Equation(5.6)} and Lemma \ref{Lemma 3.9} we have,
\begin{equation}
\label{Equation(5.7)}
    \ell_{R}\left(R/{I_{q}}\right)=\sum_{h=1}^{[\mathbb{K}_{\vartheta}: \mathbb{K}]} \# (\vartheta^{h}(R) \cap qH) - \sum_{h=1}^{[\mathbb{K}_{\vartheta}: \mathbb{K}]} \# (\vartheta^{h}(I_{q}) \cap qH)
    \end{equation}
Fix $q^{'}$ a power of $p$. I do the same calculation with respect to the $p$-system of ideals $J_{\bullet}=\{I_{q^{'}}^{[p^{e}]} \}_{e=0}^{\infty}$ then using Lemma \ref{Lemma 4.4} and Lemma \ref{Lemma 3.9}, we have\\
\begin{equation}
\label{Equation(5.8)}
\ell_{R}\left(R/{I_{q^{'}}^{[p^{e}]}}\right)=\sum_{h=1}^{[\mathbb{K}_{\vartheta}: \mathbb{K}]} \# (\vartheta^{h}(R) \cap p^{e}q^{'}H) - \sum_{h=1}^{[\mathbb{K}_{\vartheta}: \mathbb{K}]} \# (\vartheta^{h}(I_{q^{'}}^{[p^{e}]}) \cap p^{e}q^{'}H)
\end{equation}
Note that we can consider $J_{q}^{'} = {R} $ for all $q=p^{e}$. It is clearly a $p$-family of ideals. Therefore using Proposition \ref{Proposition 4.8} and  Equation \ref{Equation(5.7)}, we have
\begin{equation}
\label{Equation(5.9)}
 \lim_{q\to\infty} \frac{\ell_{R}(R/I_{q})}{q^{d}}=   [\mathbb{K}_{\vartheta}:\mathbb{K}].(\Vol_{\mathbb{R}^{d}}(\Delta(S,\vartheta(J_{\bullet}^{'}) \cap H)-\Vol_{\mathbb{R}^{d}}(\Delta(S,\vartheta(I_{\bullet}) \cap H))
\end{equation}
Similarly applying Proposition \ref{Proposition 4.8} and Equation \ref{Equation(5.8)} for the $p$-system of ideals $J_{\bullet}=\{I_{q^{'}}^{[p^{e}]} \}_{e=0}^{\infty}$, we have
\begin{equation}
\label{Equation(5.10)}
    \lim_{e\to\infty} \frac{\ell_{R}\left(R/I_{q^{'}}^{[p^{e}]}\right)}{p^{ed}(q^{'})^{d}}=   [\mathbb{K}_{\vartheta}:\mathbb{K}].(\Vol_{\mathbb{R}^{d}}(\Delta(S,\vartheta(J_{\bullet}^{'}) \cap H)-\Vol_{\mathbb{R}^{d}}(\Delta(S,\vartheta(J_{\bullet}) \cap H))
\end{equation}
Using Definition \ref{P-body}, we have 
\begin{equation}
\label{Equation(5.11)}
    \Delta(S,\vartheta(J_{\bullet}))=\bigcup_{e=0}^{\infty} \frac{1}{p^{e}} \vartheta\left(I_{q^{'}}^{[p^{e}]}\right)+\Cone(S).
\end{equation}
Now  $\vartheta\left(I_{q^{'}}^{[p^{e}]}\right) \subseteq \vartheta (I_{q^{'}p^{e}})$, as $\{I_{q}\}$ is a $p$-system of ideals and $\frac{1}{p^{e}} \vartheta\left(I_{q^{'}}^{[p^{e}]}\right) \subseteq \frac{1}{p^{e}} \vartheta(I_{q^{'}p^{e}}) \subseteq \frac{1}{p^{e}q^{'}} \vartheta(I_{q^{'}p^{e}})$. Therefore $\bigcup_{e=0}^{\infty} \frac{1}{p^{e}} \vartheta\left(I_{q^{'}}^{[p^{e}]}\right)+\Cone(S) \subseteq \bigcup_{e=0}^{\infty} \frac{1}{p^{e}q^{'}} \vartheta(I_{q^{'}p^{e}})+\Cone(S) \subseteq \bigcup_{q=1}^{\infty} \frac{1}{q} \vartheta(I_{q})+\Cone(S) = \Delta(S,\vartheta(I_{\bullet}) \cap \mathbb{H}) $. From this, we conclude\\
\begin{equation}
\label{Equation(5.12)}
    \Vol_{\mathbb{R}^{d}}(\Delta(S,\vartheta(J_{\bullet}) \cap H)) \leqslant \Vol_{\mathbb{R}^{d}}(\Delta(S,\vartheta(I_{\bullet}) \cap H))
\end{equation}
Moreover $p^{e}\vartheta(I_{q^{'}})+S \subseteq \vartheta(I_{q^{'}}^{[p^{e}]})$. Indeed for $\boldsymbol{s}_{1} \in \vartheta(I_{q^{'}})$ and $\boldsymbol{s}_{2} \in S$, we have $\vartheta(x)=\boldsymbol{s}_{1}$ for some $x \in I_{q^{'}}$ and $\vartheta(y)=\boldsymbol{s}_{2}$ for some $y \in R$. Therefore $p^{e}\boldsymbol{s}_{1} + \boldsymbol{s}_{2}= \vartheta(x^{[p^{e}]} \times y) = $ and $x^{[p^{e}]} \times y \in I_{q^{'}}^{[p^{e}]}$.

Using the observation above and Theorem \ref{Fujita type approximation}, we have for any given $\epsilon > 0$, there exist $q_{0}$ such that for all $q^{'} \geqslant q_{0}$, \\
\begin{equation}
\label{Equation(5.13)}
    \Vol_{\mathbb{R}^{d}}(\Delta(S,\vartheta(I_{\bullet}) \cap H)-\epsilon \leqslant \lim_{e\to\infty}  \frac  {\#(p^{e}\vartheta(I_{q^{'}})+S) \cap p^{e}q^{'}H)}{p^{ed}(q^{'})^{d}} \leqslant \lim_{e\to\infty}  \frac {\# (\vartheta(I_{q^{'}}^{[p^{e}]} \cap p^{e}q^{'}H)}{p^{ed}(q^{'})^{d}}
\end{equation}
Now using Equation \ref{Equation(5.12)} and \ref{Equation(5.13)} in Equation \ref{Equation(5.9)} and \ref{Equation(5.10)} we have for all $q^{'} \geqslant q_{0}$,
\begin{equation*}
\label{Equation(5.14)}
\lim_{q\to\infty} \frac{\ell_{R}(R/I_{q})}{q^{d}} \leqslant \lim_{e\to\infty} \frac{\ell_{R}(R/I_{q^{'}}^{[p^{e}]})}{p^{ed}(q^{'})^{d}}=\dfrac{e_{HK}(I_{q^{'}},R)}{(q^{'})^{d}} \leqslant \lim_{q\to\infty} \frac{\ell_{R}(R/I_{q})}{q^{d}} + \epsilon
\end{equation*}
Therefore letting $q^{'} \to \infty$, we obtain the result in OK-domain case.\\
\textbf{Step 2: Reduction to the OK-domain case:}

Since $\ell(R/J)$ is unaffected by completion for any $\mathfrak{m}$-primary ideal $J$, we may assume that $R$ is complete.\\
Using \cite[Corollary 5.18]{hernandez2018local} and Corollary \ref{Corollary 4.12} by choosing  $R=\hat{R}$ and $N=N(\hat{R})$ with the condition that $\Dim N < d$ we can therefore assume that $R$ is complete and reduced.\\
Now suppose that the minimal primes of the complete reduced ring  $R$ are  $\{P_{1},..., P_{n}\}$. Let $R_{i}=R/P_{i}$ be complete local domain for all $1 \leqslant i \leqslant n$. By Remark \ref{Remark(4.16)}, we have
\begin{equation}
\label{Equation(5.16)}
    \dfrac{e_{HK}(I_{q},R)}{(q)^{d}}=\sum_{i=1}^{n}\dfrac{e_{HK}(I_{q}R_{i},R_{i})}{(q)^{d}}
\end{equation}
We also have from Lemma \ref{Lemma 5.1}
\begin{equation}
\label{Equation(5.17)}
   \lim_{q \to \infty} \dfrac{\ell_{R}({R/I_{q}})}{q^{d}}=\sum_{i=1}^{n} \lim_{q \to \infty} \dfrac{\ell_{R_{i}}(R_{i}/I_{q}R_{i})}{q^{d}}
\end{equation}
Since each $R_{i}$ is a complete local domain and therefore an OK-domain by Example \ref{Example 3.5}, using Equation \ref{Equation(5.16)} and \ref{Equation(5.17)} we have
\begin{flalign*}
    \lim_{q \to \infty} \dfrac{\ell_{R}({R/I_{q}})}{q^{d}}  & =\sum_{i=1}^{n} \lim_{q \to \infty} \dfrac{\ell_{R_{i}}(R_{i}/I_{q}R_{i})}{q^{d}}\\
    & =\sum_{i=1}^{n} \lim_{q \to \infty}\dfrac{e_{HK}(I_{q}R_{i},R_{i})}{(q)^{d}}\\
    & =\lim_{q\to \infty}\dfrac{e_{HK}(I_{q},R)}{(q)^{d}}
    \end{flalign*}
\end{proof}

\section{Acknowledgement}
I would like to thank my advisor Dr. Jonathan Monta\~{n}o, for constant encouragement and many helpful suggestions; specifically, our Starbucks discussions made this a journey to remember. The author was partially supported by NSF Grant DMS $\#2001645$.
\bibliographystyle{plain}
\bibliography{citation}
\end{document}